\documentclass[a4paper,11pt,oneside,headsepline]{scrartcl}

\usepackage{ifdraft}

\usepackage[utf8]{inputenc}
\usepackage[T1]{fontenc}

\usepackage{lmodern}
\usepackage{fourier}

\usepackage{amssymb, amsfonts}

\usepackage{mathrsfs} 
\usepackage{dsfont}
\usepackage[usenames,dvipsnames]{xcolor}
\usepackage{lettrine}
\usepackage{url}

\usepackage{stmaryrd}

\usepackage[english]{babel}

\usepackage{scrpage2}
\usepackage{wrapfig}
\usepackage[compact]{titlesec}
\usepackage{footmisc}
\usepackage{needspace}

\usepackage{amsmath}
\usepackage[thmmarks, amsmath, amsthm]{ntheorem}

\usepackage[style=alphabetic,firstinits,url=false,sortcites=true,maxbibnames=99,backend=bibtex]{biblatex}
\bibliography{bibliography.bib}

\ifdraft{{}}
  {\usepackage[pdftex,
               pdfborder={0 0 0}, colorlinks=true,
               linkcolor=BrickRed, citecolor=ForestGreen, urlcolor=RoyalBlue]{hyperref}}

\ifdraft{\usepackage[colorinlistoftodos]{todonotes}}{}

\usepackage{booktabs}
\usepackage[inline]{enumitem}
\usepackage{float}
\usepackage{multicol}
\usepackage{tikz}
\usetikzlibrary{matrix,arrows,positioning}

\ifdraft{\date{\today}}{\date{}}

\KOMAoptions{DIV=13}

\clearscrheadings
\automark[section]{subsection}

\renewcommand{\subsectionmark}[1]{}

\lohead{{\small \headertitle}}
\rohead{{\small \headerauthors}}

\newenvironment{plainfootnotes}{
  \deffootnote[0em]{0em}{0em}{}
}{
  \deffootnote[1em]{1.5em}{1em}{\textsuperscript{\thefootnotemark}}
}

\newif\ifsubsectionstylePrefixParagraph
\newif\ifsubsectionstyleRunin

\subsectionstylePrefixParagraphtrue
\subsectionstyleRunintrue

\titleformat{\section}[hang]
{\Large\sffamily\bfseries}
{\thesection\hspace{0.25em}{}}{0.25em}{}

\ifsubsectionstyleRunin
\ifsubsectionstylePrefixParagraph

\titleformat{\subsection}[runin]
{\normalfont\bfseries}
{\S\hspace{.1em}\thesubsection}{0.35em}{}[.\hspace*{.5em}]

\else

\titleformat{\subsection}[runin]
{\normalfont\bfseries}
{\thesubsection\hspace{.1em}|}{0.25em}{}[.\hspace*{.5em}]

\fi
\else
\ifsubsectionstylePrefixParagraph

\titleformat{\subsection}[wrap]
{\normalfont\bfseries\selectfont\filright}
{\S\thesubsection}{.35em}{}
\titlespacing{\subsection}
{16pc}{1.5ex plus .1ex minus .2ex}{1pc}

\else

\titleformat{\subsection}[wrap]
{\normalfont\bfseries\selectfont\filright}
{\thesubsection\hspace{.1em}|}{.25em}{}
\titlespacing{\subsection}
{16pc}{1.5ex plus .1ex minus .2ex}{1pc}

\fi
\fi

\ifsubsectionstylePrefixParagraph

\titleformat{\subsubsection}[runin]
{\normalfont\bfseries}
{\S\hspace{.1em}\thesubsubsection}{0.35em}{}[.]

\else

\titleformat{\subsubsection}[runin]
{\normalfont\bfseries}
{\thesubsubsection\hspace{.1em}|}{0.25em}{}[.]
  
\fi

\titleformat{\paragraph}[runin]
{\normalfont\itshape}
{\S\hspace{.1em}\theparagraph}{0.35em}{}[.]

\AtEveryBibitem{\clearfield{isbn}}
\AtEveryBibitem{\clearlist{language}}
\AtEveryBibitem{\clearfield{pages}}

\newenvironment{enumeratearabic}{
\begin{enumerate}[label=(\arabic*), leftmargin=0pt,labelindent=1.5em,itemindent=!]
}{
\end{enumerate}
}

\newenvironment{enumeratearabic*}{
\begin{enumerate*}[label=(\arabic*)] 
}{
\end{enumerate*}
}

\newenvironment{enumerateroman*}{
\begin{enumerate*}[label=(\roman*)] 
}{
\end{enumerate*}
}

\numberwithin{equation}{section}

\newtheorem{theoremcounter}{theoremcounter}[section]

\theoremstyle{plain}

\newtheorem{lemma}[theoremcounter]{Lemma}
\newtheorem{proposition}[theoremcounter]{Proposition}

\theoremnumbering{Roman}
\newtheorem{maintheoremcounter}{maintheoremcounter}

\newtheorem{maintheorem}[maintheoremcounter]{Theorem}

\theoremstyle{definition}

\newtheorem{definition}[theoremcounter]{Definition}

\theoremstyle{remark}

\newtheorem{remark}[theoremcounter]{Remark}

\newtheorem*{mainremark}{Remark}

\newtheorem*{remarkcomputation}{Computation}

\ifdraft{
 \newcommand{\texpdf}[2]{#1}
}{
 \newcommand{\texpdf}[2]{\texorpdfstring{#1}{#2}}
}

\newcommand{\tx}{\ensuremath{\text}}

\newcommand{\tbf}{\bfseries}

\newcommand{\thdash}{\nbd th}

\newcommand{\nbd}{\nobreakdash-\hspace{0pt}}

\newcommand{\bboard}{\ensuremath{\mathbb}}

\renewcommand{\frak}{\ensuremath{\mathfrak}}

\newcommand{\bbM}{\ensuremath{\bboard M}}

\newcommand{\bbP}{\ensuremath{\bboard P}}

\newcommand{\frakg}{\ensuremath{\frak{g}}}

\newcommand{\frakk}{\ensuremath{\frak{k}}}

\newcommand{\frakm}{\ensuremath{\frak{m}}}

\newcommand{\rmt}{\ensuremath{\mathrm{t}}}

\newcommand{\rmA}{\ensuremath{\mathrm{A}}}

\newcommand{\rmF}{\ensuremath{\mathrm{F}}}
\newcommand{\rmG}{\ensuremath{\mathrm{G}}}

\newcommand{\rmL}{\ensuremath{\mathrm{L}}}
\newcommand{\rmM}{\ensuremath{\mathrm{M}}}

\newcommand{\rmP}{\ensuremath{\mathrm{P}}}

\newcommand{\rmR}{\ensuremath{\mathrm{R}}}

\newcommand{\wtd}{\widetilde}
\newcommand{\ov}{\overline}

\newcommand*{\longhookrightarrow}{\ensuremath{\lhook\joinrel\relbar\joinrel\rightarrow}}

\newcommand{\ra}{\ensuremath{\rightarrow}}
\newcommand{\hra}{\ensuremath{\hookrightarrow}}

\newcommand{\lhra}{\ensuremath{\longhookrightarrow}}

\newcommand{\ZZ}{\ensuremath{\mathbb{Z}}}

\newcommand{\RR}{\ensuremath{\mathbb{R}}}
\newcommand{\CC}{\ensuremath{\mathbb{C}}}

\renewcommand{\Re}{\ensuremath{\mathrm{Re}}}
\renewcommand{\Im}{\ensuremath{\mathrm{Im}}}

\newcommand{\sgn}{\ensuremath{\mathrm{sgn}}}

\newcommand{\Hom}{\ensuremath{\mathop{\mathrm{Hom}}}}

\newenvironment{psmatrix}{\left(\begin{smallmatrix}}{\end{smallmatrix}\right)}

\newcommand{\Mat}[1]{\ensuremath{\mathrm{Mat}_{#1}}}

\newcommand{\GL}[1]{\ensuremath{\mathrm{GL}_{#1}}}

\newcommand{\SL}[1]{\ensuremath{\mathrm{SL}_{#1}}}

\newcommand{\Sp}[1]{\ensuremath{\mathrm{Sp}_{#1}}}

\newcommand{\U}[1]{\ensuremath{\mathrm{U}_{#1}}}

\newcommand{\T}{\ensuremath{\rmt}}
\newcommand{\rT}{\ensuremath{\,{}^\T\!}}

\newcommand{\tr}{\ensuremath{\mathrm{tr}}}

\renewcommand{\det}{\ensuremath{\mathrm{det}}}
\newcommand{\sym}{\ensuremath{\mathrm{sym}}}

\newcommand{\HS}{\mathbb{H}}

\newcommand{\gKmodule}{$(\frakg,K)$\nbd module}
\newcommand{\gKmodules}{$(\frakg,K)$\nbd modules}

\newcommand{\sk}{\ensuremath{\mathrm{sk}}}

\newcommand{\ppartial}[1]{\ensuremath{\partial_{#1}}}

\newcommand{\headertitle}{{\normalfont%
  Siegel-Poincar\'e Series
}}
\newcommand{\headerauthors}{
  K.~Bringmann,
  O.~K.~Richter,
  M.~Westerholt-Raum
}

\begin{document}

\begin{plainfootnotes}
\begin{flushleft}
{\fontfamily{lms}\sffamily
  \hspace{20pt}{\huge%
  Almost holomorphic Poincar\'e series corresponding to}
  \\\hspace{20pt}{\huge%
  products of harmonic Siegel-Maass forms
  }
}
\\[.6em]\hspace{20pt}{\large%
  Kathrin Bringmann%
  \footnote{The research of the first author was supported by the Alfried Krupp Prize for Young University Teachers of the Krupp foundation and the research leading to these results has received funding from the European Research Council under the European Union's Seventh Framework Programme (FP/2007-2013) / ERC Grant agreement n.\ 335220 - AQSER.},
  Olav~K.~Richter
  \footnote{The second author was partially supported by Simons Foundation Grant \#200765.}, and
  Martin Westerholt-Raum%
  \footnote{The third author was partially supported by Vetenskapsr\aa det Grant~2015-04139.}
}
\\[1.2em]
\end{flushleft}
\end{plainfootnotes}

\thispagestyle{scrplain}


{\small
\noindent
{\tbf Abstract:}
We investigate Poincar\'e series, where we average products of terms of Fourier series of real-analytic Siegel modular forms. There are some (trivial) special cases for which the products of terms of Fourier series of elliptic modular forms and harmonic Maass forms are almost holomorphic, in which case the corresponding Poincar\'e series are almost holomorphic as well. In general this is not the case. The main point of this paper is the study of Siegel-Poincar\'e series of degree~$2$ attached to products of terms of Fourier series of harmonic Siegel-Maass forms and holomorphic Siegel modular forms. We establish conditions on the convergence and nonvanishing of such Siegel-Poincar\'e series.  We surprisingly discover that these Poincar\'e series are almost holomorphic Siegel modular forms, although the product of terms of Fourier series of harmonic Siegel-Maass forms and holomorphic Siegel modular forms (in contrast to the elliptic case) is not almost holomorphic.  Our proof employs tools from representation theory. In particular, we determine some constituents of the tensor product of Harish-Chandra modules with walls.
\\[.35em]
\textsf{\textbf{%
  almost holomorphic modular forms%
}}
\hspace{0.3em}{\tiny$\blacksquare$}\hspace{0.3em}%
\textsf{\textbf{%
  Siegel modular forms%
}}
\hspace{0.3em}{\tiny$\blacksquare$}\hspace{0.3em}%
\textsf{\textbf{%
  Harish-Chandra modules%
}}
\\[0.15em]
\noindent
\textsf{\textbf{%
  MSC Primary:
  11F46%
}}
\hspace{0.3em}{\tiny$\blacksquare$}\hspace{0.3em}%
\textsf{\textbf{%
  MSC Secondary:
  11F30, 11F37, 11F70
}}
}


\vspace{1.5em}


\Needspace*{4em}
\addcontentsline{toc}{section}{Introduction}
\markright{Introduction}
\lettrine[lines=2,nindent=.2em]{\tbf M}{odular} forms have a rich history with many beautiful applications in different sciences.  Restricting our attention to cusp forms, we can distinguish two classes of modular forms, which behave drastically differently: Maass cusp forms and almost holomorphic modular forms.  Almost holomorphic forms modular forms were introduced independently by Shimura~\cite{shimura-1987} and Kaneko and Zagier~\cite{kaneko-zagier-1995}. 
Shimura initiated their study, because of their special arithmetic properties. Kaneko and Zagier's paper was stimulated by Dijkgraaf~\cite{dijkgraaf-1995}, who conjectured a relation between almost holomorphic modular forms and numbers of topologically inequivalent branched covers of an elliptic curve. Research on the latter aspect has been particularly successful, establishing connections to Kac-Moody algebras and related representation theoretic objects~\cite{bloch-okounkov-2000,eskin-okounkov-pandharipande-2008}.

Poincar\'e series play a major role in the theory of automorphic forms (for example, see \cite{fay-1977,bringmann-ono-2006,duke-imamoglu-toth-2011} among many others). It is not difficult to construct almost holomorphic Poincar\'e series by observing that almost holomorphic modular forms vanish under a power of the lowering operator~$\rmL: = -2 i y^2 \ppartial{\ov\tau}$, where throughout $\tau=x+iy\in\HS$.  Specifically, the elliptic Poincar\'e series $\sum_\gamma (y^{-d} e^{2 \pi i\, n\tau}) \big|_k\, \gamma$ converges for  $n, d \in \ZZ$, $n>0$, $k>2$, $0 \le d < \frac{k}{2} - 1$, and it is almost holomorphic, since $\rmL$ is equivariant with respect to the usual slash action $|_k$ and $\rmL^{d+1} (y^{-d} e^{2 \pi i\, n\tau})=0$.  Slightly more general, if $s\in\CC$, then Selberg's~\cite{selberg-1965} Poincar\'e series $\sum_\gamma \big( ( y^s ) \cdot ( e^{2 \pi i\, n \tau} ) \big) \big|_k\, \gamma$ corresponds to a product of terms of Fourier series.  The factor $y^s$ is the $0$\thdash\ term of the Fourier series of a (weight $0$) Eisenstein series, and the factor $e^{2 \pi i\, n \tau}$ is a term of Fourier series of a (weight~$k$) holomorphic modular form. In general, if the product of two terms of Fourier series is not almost holomorphic, then the associated Poincar\'e series is not almost holomorphic either, provided that it is nonzero. We discuss two examples in more detail in Section~\ref{ssec:elliptic:spectral-decomposition}.

Siegel modular forms impact many different areas of mathematics: Algebraic and arithmetic geometry, invariant theory, representation theory, quantum theory, and conformal field theory, for example. Almost holomorphic Siegel modular forms were recently classified in~\cite{klemm-poretschkin-schimannek-raum-2015,pitale-saha-schmidt-2015}, and they play an important role in the context of mirror symmetry in~\cite{klemm-poretschkin-schimannek-raum-2015}.

The purpose of this paper is to construct almost holomorphic Siegel-Poincar\'e series of degree~$2$, where we average products of terms of Fourier series of (harmonic) Siegel-Maass forms, but where these products are not almost holomorphic themselves. More specifically, in~\eqref{eq:siegel:poincare-series} we define the Siegel-Poincar\'e series
\begin{gather*}
  \bbP^{(2)}_{k,\ell;T,T'} (Z)
=
  \sum_M \Big( \Psi_k(T;Z) \cdot \Phi_\ell(T';Z) \Big) \big|_{k+\ell}\, M
\text{,}
\end{gather*}
where $k$, $\ell$ are positive even integers, $T$, $T'$ are positive definite and symmetric $2\times 2$ matrices, and where $\Psi_k(T;Z)$ and $\Phi_\ell(T';Z)$ are the $T$\thdash\ and $T'$\thdash\ terms of Fourier series of a weight $k$ harmonic Siegel-Maass form (as in~\cite{bringmann-raum-richter-2011,raum-2012d}) and weight~$\ell$ holomorphic Siegel modular form, respectively.  We apply an estimate of~\cite{shimura-1982} to prove that $\bbP^{(2)}_{k,\ell;T,T'}$ converges for all $\ell$ large enough. Moreover, we employ a result of~\cite{kowalski-saha-tsimerman-2011} to show that $\bbP^{(2)}_{k,\ell;T,T'}$ does not vanish identically for all $\ell$ large enough.  The following theorem is our main result.
\begin{maintheorem}
\label{thm:maintheorem}
Assume the generalized Ramanujan conjecture for $\GL{4}$. The function $\Psi_k(T;Z) \cdot \Phi_\ell(T';Z)$ is \emph{not} almost holomorphic. If $\ell \ge 6 + 2b - k$, where $b > 0$ is defined in \eqref{eq:shimuras-A-and-B}, then $\bbP^{(2)}_{k,\ell;T,T'}$ converges and is almost holomorphic.
\end{maintheorem}

\begin{mainremark}
\begin{enumeratearabic}
\item
Poincar\'e series are the starting point for Kuznetsov-type~\cite{kuznecov-1981} trace formulas. It would be interesting to determine the Fourier series coefficients and the spectral decomposition of $\bbP^{(2)}_{k,\ell;T,T'}$ to discover a novel Kuznetsov-type trace formula for Siegel modular forms.

\item
It is possible to define $\bbP^{(2)}_{k,\ell;T,T'}$ for indefinite $T$. One can show that these Poincar\'e series converge for sufficiently large $\ell$ and that they are almost holomorphic. However, unfolding the Petersson scalar product against an arbitrary almost holomorphic Siegel modular form yields that they vanish identically.
\end{enumeratearabic}
\end{mainremark}

The paper is organized as follows.  In Section~\ref{sec:elliptic}, we illustrate how almost holomorphic Poincar\'e series arise in the setting of elliptic modular forms. In Section~\ref{sec:siegel}, we review real-analytic Siegel modular forms.  In particular, we recall the notions of almost holomorphic Siegel modular forms and harmonic Siegel-Maass forms.  In Section~\ref{sec:poincare-series}, we define the Siegel-Poincar\'e series $\bbP^{(2)}_{k,\ell;T,T'}$, and we determine conditions on its convergence and nonvanishing. In Section~\ref{sec:harish-chandra-modules}, we prove Theorem~\ref{thm:maintheorem} using tools from real representation theory: The Poincar\'e series $\bbP^{(2)}_{k,\ell;T,T'}$ yields a cuspidal automorphic representation for $\mathrm{PGSp}_n$. The \gKmodule\ attached to its component at the infinite place embeds into the tensor product of two Harish-Chandra modules generated by $\Psi_k$ and $\Phi_\ell$. This puts severe restrictions on \gKmodules\ arising in our context. Specifically, we employ Mui\'c's~\cite{muic-2009} study of decompositions of generalized principal series to control Harish-Chandra parameters of \gKmodules\ that have ``walls'' in their $K$\nbd type support. In addition, $\bbP^{(2)}_{k,\ell;T,T'}$ gives rise to a scalar $K$\nbd type.  Consequently, we can invoke Arthur's~\cite{arthur-2013} endoscopic classification of representations of~$\Sp{2}(\RR)$ to narrow down possibilities to holomorphic (limits of) discrete series.


\section{Poincar\'e series for elliptic modular forms}
\label{sec:elliptic}

We consider cuspidal Poincar\'e series in the case of elliptic modular forms to demonstrate possible phenomena, and to explain available tools to investigate such series. From the introduction, recall the almost holomorphic Poincar\'e series attached to the function
\begin{gather}
\label{eq:elliptic:almost-holomorphic-fourier-coefficients}
  \phi_{k[d]}(n; \tau)
:=
  y^{-d} e^{2 \pi i\, n \tau}
\,\tx{,}\quad
  n > 0
\tx{.}
\end{gather}
Note that $\phi_{k[d]}$ does not depend on~$k$.  Nevertheless, we also include it in the notation to indicate that it is a typical term of a weight~$k$ almost holomorphic modular form. If $d=0$, then we write $\phi_{k}(n; \tau) := \phi_{k[0]}(n; \tau)$.

To define elliptic Poincar\'e series, set $\Gamma^{(1)}_{\infty} := \big\{ \pm \begin{psmatrix}1 & b \\ 0 & 1\end{psmatrix}\tx{,}\, b\in\ZZ \big\}$, and recall the elliptic slash action $(f |_k\, \gamma )(\tau) := (c \tau + d)^{-k} f(\frac{a \tau + b}{c \tau + d})$ for $f :\, \HS \ra \CC$, $k \in \ZZ$, and $\gamma = \begin{psmatrix} a & b \\ c & d \end{psmatrix} \in \SL{2}(\ZZ)$. Products of almost holomorphic functions are almost holomorphic. Hence it is trivial that the Poincar\'e series
\begin{gather}
\label{eq:elliptic:almost-holomorphic-poincare-series}
  \rmP^{(1)}_{k[d],\ell;n,m}(\tau)
:=
  \sum_{\gamma \in \Gamma^{(1)}_\infty \backslash \SL{2}(\ZZ)}
  \big( \phi_{k[d]}(n; \tau) \, \phi_\ell(m; \tau) \big) \big|_{k+\ell}\, \gamma
\end{gather}
is almost holomorphic, provided that it converges, which is the case if $k + \ell > 2 + 2d$. In the next two sections, we will demonstrate that analogous Poincar\'e series are, in general, not almost holomorphic.

\subsection{Poincar\'e series that are not almost holomorphic}

Consider a typical term of the nonholomorphic part of a weight~$k$ harmonic weak Maass form:
\begin{gather}
\label{eq:elliptic:harmonic-fourier-coefficients}
  \wtd\psi_{k}(n; \tau)
:=
  \Gamma(1-k, 4 \pi |n| y)\, e^{2 \pi i\, n \tau}
\,\tx{,}\quad
  k \le 0,\, n < 0
\tx{,}
\end{gather}
where $\Gamma$ is the usual incomplete gamma function. Note that $\wtd\psi_k(n;\,\cdot\,)$ decays rapidly toward infinity, but the Poincar\'e series $\sum_\gamma \wtd\psi_k(n;\tau) \big|_k \gamma$ does not converge, due to its behavior as $y \ra 0$. For the remainder of this section, let $\ell \in \ZZ$ such that $k + \ell > 2$, in which case the Poincar\'e series
\begin{gather}
\label{eq:elliptic:harmonic-poincare-series}
  \wtd\bbP^{(1)}_{k,\ell;n,m}(\tau)
:=
  \sum_{\gamma \in \Gamma^{(1)}_\infty \backslash \SL{2}(\ZZ)}
  \big( \wtd\psi_{k}(n; \tau) \, \phi_\ell(m; \tau) \big) \big|_{k+\ell}\, \gamma
\end{gather}
converges.

Our main tool in this paper---see  Section~\ref{sec:harish-chandra-modules}---is the theory of Harish-Chandra modules (i.e., real representation theory). We will apply it to the case of elliptic modular forms in Section~\ref{ssec:elliptic:harish-chandra}. Guided by the emphasis of real representation theory on linear differential operators, we consider the lowering  operator~$\rmL$ instead of Bruinier and Funke's~\cite{bruinier-funke-2004} $\xi$\nbd operator. The $\xi$\nbd operator and $\rmL$ map smooth functions on~$\HS$ to smooth functions on~$\HS$, and are defined by
\begin{gather*}
  \xi_k(f)
:=
  2 i y^k \ov{ \ppartial{\ov\tau}\, f }
\quad\tx{and}\quad
  \rmL
:=
  \rmL_k
:=
  -2 i y^2 \ppartial{\ov\tau}
\,\tx{,}
\end{gather*}
where $\ppartial{\ov\tau}\,f = \partial f \big\slash \partial \ov\tau$ and $\ppartial{\tau}\,f = \partial f \big\slash \partial \tau$\vspace{-.3em}. The space~$\bbM^{(1)}_k$ of harmonic Maass forms is mapped by $\xi_k$ to $\rmM^{(1)}_{2-k}$, the space of holomorphic modular forms. Equivalently, it is mapped by $\rmL_k$ to $y^{2-k} \ov{\rmM}^{(1)}_{2-k}$, which is the space of antiholomorphic modular forms of antiholomorphic weight~$2-k$ normalized by $y^{2-k}$ to holomorphic weight~$k-2$. Terms of Fourier series are related by the formula $\rmL_k \big( \wtd\psi_k(n; \tau) \big) = (4 \pi |n|)^{1-k}\, \wtd\phi_{k-2}(-n;\tau)$ with
\begin{gather}
\label{eq:elliptic:antiholomorphic-fourier-coefficients}
  \wtd\phi_k(n; \tau)
:=
  y^{-k}\ov{\phi_{-k}(n; \tau)} 
=
  y^{-k} e^{- 2 \pi i\, n \ov{\tau}}
\,\tx{,}\quad
  k < 0,\, n < 0
\tx{.}
\end{gather}
The image of~$\wtd\bbP^{(1)}_{k,\ell;n,m}$ under $\rmL_k$ is therefore a scalar multiple of
\begin{gather}
\label{eq:elliptic:antiholomorphic-poincare-series}
  \wtd{\rmP}^{(1)}_{k-2,\ell;n,m}(\tau)
:=
  \sum_{\gamma \in \Gamma^{(1)}_\infty \backslash \SL{2}(\ZZ)}
  \big( \wtd\phi_{k-2}(n; \tau) \, \phi_\ell(m; \tau) \big) \big|_{k-2+\ell}\, \gamma
\tx{.}
\end{gather}

For simplicity, we focus on $|m| > |n|$.
%
In the next section, we find that \emph{neither}~$\wtd{\rmP}^{(1)}_{k,\ell;n,m}$ nor~$\wtd{\bbP}^{(1)}_{k,\ell;n,m}$ is almost holomorphic.

\subsection{Spectral decomposition of Poincar\'e series}
\label{ssec:elliptic:spectral-decomposition}

It is an important and nontrivial task to find the spectral decomposition of an automorphic form. The spectrum of the weight~$k$ hyperbolic Laplace operator on $L^2(\SL{2}(\ZZ) \backslash \HS,\, |_k)$ consists of four different contributions:
\begin{enumeratearabic*}
\item
Eisenstein series~$E_{k,s}$ with spectral parameter $s = 1/2 + it$, $t \in \RR$.

\item
The residual spectrum, which arises from unary theta series.

\item
Almost holomorphic and almost antiholomorphic cusp forms.

\item
Proper Maass cusp forms.
\end{enumeratearabic*}
The Poincar\'e series that we consider are all cuspidal, and so Eisenstein series do not occur in their spectral decomposition. Furthermore, the residual spectrum occurs only for half-integral weight, and therefore it does not contribute to the weights that we treat. Consequently, Parseval's equation implies that if $f$ is one of the Poincar\'e series, then
\begin{gather*}
  f
=
  \sum_j \langle g_j, f \rangle g_j
  +
  \sum_j \langle u_j, f \rangle u_j
\end{gather*}
where $g_j$ runs through a (finite) complete orthonormal set of almost holomorphic and antiholomorphic modular forms, and $u_j$ runs through an (infinite) complete orthonormal set of proper Maass cusp forms.

We now show that $\wtd\rmP^{(1)}_{k,\ell;n,m}$ in~\eqref{eq:elliptic:antiholomorphic-poincare-series} and $\wtd\bbP^{(1)}_{k,\ell;n,m}$ in~\eqref{eq:elliptic:harmonic-poincare-series} are not almost holomorphic. More precisely, we show that both have spectral expansions with infinite support. The relation
\begin{gather*}
  \rmL_k \Big( \wtd{\bbP}^{(1)}_{k,\ell;n,m}(\tau) \Big)
=
  (4 \pi |n|)^{1-k}\;
  \wtd{\rmP}^{(1)}_{k-2,\ell;n,m}(\tau) 
\end{gather*}
allows us to examine only~$\wtd\rmP^{(1)}_{k,\ell;n,m}$.

To exhibit the spectral decomposition, suppose that $u$ is a Maass cusp form with spectral parameter~$s$ with $\Re(s) > \frac{1}{2}$. Without loss of generality, we can raise its weight to weight~$k+\ell$, such that its spectral parameter with respect to the weight~$k+\ell$ Laplace operator is~$s - (k+\ell)/2$. The Fourier series expansion of~$u$ has the form
\begin{gather*}
  u(\tau)
=
  \sum_{n \in \ZZ \setminus \{0\}}\!\!
  c(u;\, n)\,
  (4 \pi |n| y)^{-\frac{k+\ell}{2}}
  W_{\sgn(n) \frac{k+\ell}{2},\, s-\frac{1}{2}}(4 \pi |n| y)
  e^{2 \pi i\, n x}
\,\tx{,}
\end{gather*}
where $W_{\mu,\nu}$ denotes the usual $W$-Whittaker function. We unfold the Petersson scalar product of the Poincar\'e series $\wtd\rmP^{(1)}_{k,\ell; n,m}$ against $u$, and obtain by 7.621.3 of~\cite{gradshteyn-ryzhik-2007}
that $\langle u,\, \wtd\rmP^{(1)}_{k,\ell;n,m} \rangle$ equals
\begin{multline*}
  c(u;\, m - n)
  \int_0^\infty
  y^{-k} e^{-2\pi (m-n) y}\,
  (4 \pi\, (m-n) y)^{-\frac{k+\ell}{2}} W_{\frac{k+\ell}{2},\, s-\frac{1}{2}}(4 \pi (m-n) y)
  \;\frac{d\!y}{y^{2-(k+\ell)}}
\\
=
  c(u;\, m - n)\;
  \big( 4 \pi (m - n) \big)^{1 - l}\,
  \frac{\Gamma\big(\frac{\ell-k}{2} + s - 1\big) \Gamma\big(\frac{\ell-k}{2} - s\big)}{\Gamma(-k)}\;
  {}_2 \rmF_1\big( \tfrac{\ell-k}{2} + s - 1, s - \tfrac{k+\ell}{2},\, -k;\,0 \big)
\tx{.}
\end{multline*}
The hypergeometric series evaluated at~$0$ equals~$1$. The gamma-factors do not vanish and all Maass  cusp forms with $c(u;\,m - n) \ne 0$ contribute to the spectral expansion of~$f$. By Theorem~1.2 of~\cite{matz-templier-2015}, there are infinitely many such Maass cusp forms.

\subsection{Almost holomorphic Poincar\'e series}
\label{ssec:elliptic:almost-holomorphic}

In this section, we suggest another Poincar\'e series that is almost holomorphic. In light of what we discuss in Section~\ref{ssec:elliptic:harish-chandra}, it is natural to replace $\wtd\psi_k$ by
\begin{gather}
\label{eq:elliptic:anti-harmonic-fourier-coefficients}
  \psi_{k}(n; \tau)
:=
  y^{-k}\, \Gamma(1+k, 4 \pi n y)\, e^{2 \pi i\, n \ov{\tau}}
\,\tx{,}\quad
  k \ge 0,\, n > 0
\tx{.}
\end{gather}
It is a term of the Fourier series of an ``anti-harmonic'' Maass form in $y^{-k} \ov\bbM^{(1)}_{-k}$, which is a space of functions that are mapped by the raising operator
\begin{gather}
  \rmR_k
:=
  2 i \ppartial{\tau} + k y^{-1}
\end{gather}
to holomorphic modular forms of weight~$k+2$.

Define the Poincar\'e series
\begin{gather}
\label{eq:elliptic:anti-harmonic-poincare-series}
 \bbP^{(1)}_{k,\ell;n,m}(\tau)
:=
  \sum_{\gamma \in \Gamma^{(1)}_\infty \backslash \SL{2}(\ZZ)}
  \big( \psi_{k}(n; \tau) \, \phi_\ell(m; \tau) \big) \big|_{k+\ell}\, \gamma
\tx{,} 
\end{gather}
which converges if $\ell - k > 2$. Observe that
\begin{gather*}
  \psi_{k}(n; \tau) \, \phi_\ell(m; \tau)
=
  y^{-k} \Gamma(1+k, 4 \pi n y) e^{2 \pi i\, n \ov{\tau}}
  e^{2 \pi i\, m \tau}
=
  p(y^{-1}) e^{2 \pi i\, (n + m) \tau}
\,\tx{,}
\end{gather*}
where $p$ is a polynomial of degree~$k$. Applying the lowering operator $k+1$ times annihilates this product, i.e., $\bbP^{(1)}_{k,\ell;n,m}$ is almost holomorphic.

Finally, it is easy to see that $\rmP^{(1)}_{k[d],\ell; n, m}$ (defined in~\eqref{eq:elliptic:almost-holomorphic-poincare-series}) is in the kernel of $L^{d+1}$, i.e., it is also almost holomorphic.

\subsection{Harish-Chandra modules}
\label{ssec:elliptic:harish-chandra}

The spectral decomposition of an automorphic form implies a decomposition of the associated Harish-Chandra module. Vice versa, one can deduce from the decomposition of the Harish-Chandra module attached to an automorphic form which parts of the spectrum contribute to its spectral expansion. In particular, it is possible to infer from the Harish-Chandra module alone whether an automorphic form is almost holomorphic. In this section, we give some details and examples. A precise statement in the case of automorphic forms for $\Sp{2}(\RR)$ is presented in Section~\ref{sec:harish-chandra-modules}.

We reconsider the spectral decomposition of the Poincar\'e series~$\rmP^{(1)}_{k[d],\ell; n, m}$ and~$\bbP^{(1)}_{k,\ell; n,n'}$ in light of Harish-Chandra modules. The Poincar\'e series in~\eqref{eq:elliptic:harmonic-poincare-series} cannot be completely analyzed and their \gKmodules\ are not Harish-Chandra modules. An excellent treatment of Harish-Chandra modules can, for example, be found in~\cite{wallach-1988,knapp-2001}. To accommodate the classically inclined reader, we suggest a rather simple schematic way of thinking of Harish-Chandra modules in terms of lowering and raising operators; but for simplicity we suppress details on the correspondence of Harish-Chandra modules and the $\CC[\rmL,\rmR]$-modules that we employ in this section. We will be more precise in Section~\ref{sec:harish-chandra-modules}, when we discuss the case of $\Sp{2}(\RR)$.

Given an even weight~$k$ and a function~$f$, we study the action of $\CC[\rmL, \rmR]$ where $\rmL$ and $\rmR$ are abstract lowering operators on the pair $(f,k)$. The action is precisely given by $\rmL\, (f,k) = (\rmL_k\, (f), k-2)$ and $\rmR\, (f,k) = (\rmR_k\, (f), k+2)$, where $\rmL_k$ and $\rmR_k$ are the usual lowering and raising operators, as before given by
\begin{gather*}
  \rmL_k
:=
  -2 i y^2 \ppartial{\ov\tau}
\quad\tx{and}\quad
  \rmR_k
:=
  2 i \ppartial{\tau} + k y^{-1}
\,\tx{.}
\end{gather*}

By viewing the second component of $(f,k)$ as a grading, it makes sense to speak of the graded module $\CC[\rmL, \rmR]\, f$, and study for which $k$ its $k$\thdash\ graded component is nonzero. We are further interested in the action of the lowering and raising operators on the graded components of $\CC[\rmL, \rmR]\, (f,k)$. We illustrate the typical behavior on the function $y^s$ with $s \in \CC$. We clearly have $\rmL_k\, (y^s) = s y^{s+1}$ and $\rmR_k\, (y^s) = (s+k) y^{s-1}$. If $s = 0$ or $s = -k$, then $y^s$ vanishes under $\rmL_k$ or $\rmR_k$, respectively. Rewriting this in terms of pairs $(f,k)$, we find that $\rmL (y^s, 0) = ( s y^{s+1}, -2)$ and $\rmR (y^s, 0) = (s y^{s-1}, 2)$.

We now exhibit the behavior of terms of the Fourier series in~\eqref{eq:elliptic:almost-holomorphic-fourier-coefficients}, \eqref{eq:elliptic:harmonic-fourier-coefficients}, \eqref{eq:elliptic:antiholomorphic-fourier-coefficients}, and~\eqref{eq:elliptic:anti-harmonic-fourier-coefficients} under lowering and raising operators. For any $n$, we have $\rmL (\phi_k(n;\tau), k) = 0$, and no power of $\rmR$ annihilates $(\phi_k(n;\tau), k)$. The situation for $\psi_k(n;\tau), k)$ is similar. We have $\rmL^{k+1} (\psi_k(n;\tau), k) = (0, -k-2)$. In addition, $\rmL \rmR (\psi_k(n;\tau), k) = 0$. The behavior of $\wtd\phi_k$ and $\wtd\psi$ is analogous.

We now introduce a schematic way to describe these modules. Every graded component which is nonzero corresponds to a filled circle. Graded components that are zero correspond to (small) circles. Both are placed on a line to emphasize the $2\ZZ$\nbd grading. We further encircle one dot that corresponds to the graded component which we are focusing on. For example, in the case of $(y^{-d},0)$ we would encircle the $0$\thdash\ graded component, and when considering $(\phi_k,k)$ we encircle the $k$\thdash\ one. Finally, we separate $\CC[\rmR,\rmL]$\nbd submodules by vertical lines decorated with an arrow. It means that applying the lowering operator (if the arrow points rightwards) or the raising operator (if the arrow points leftwards) applied to this graded component equals zero. In addition, we insert vertical dashed lines to indicate the relation to walls in so-called principal series representations. Here are diagrams for $\phi_{k[d]}$, $\psi_k$, $\wtd\phi_k$, and $\wtd\psi_k$.
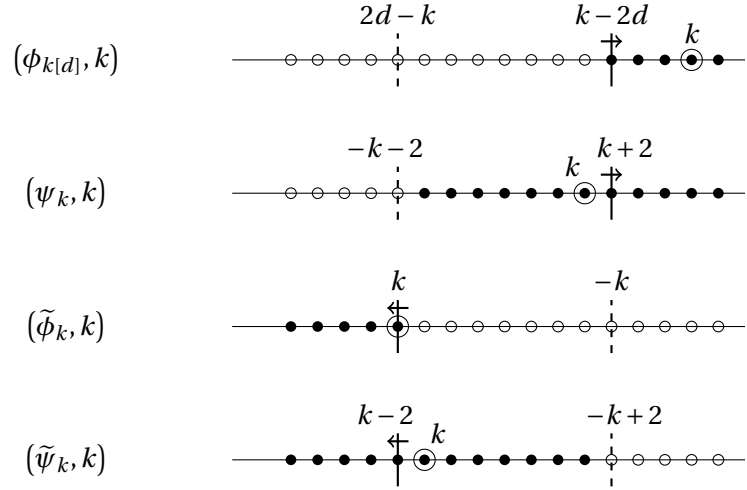
\begin{center}
\begin{tikzpicture}
\draw[draw=white] (-4.5,0.3) rectangle (-1,-0.3) node[pos=0.5] {$\big(\phi_{k[d]},k\big)$};

\foreach \ix in {0,...,3}{
  \draw (\ix 6pt, 0pt) circle[radius=2pt];
}
\foreach \ix in {4,...,11}{
  \draw (\ix 6pt, 0pt) circle[radius=2pt];
}
\foreach \ix in {12,...,16}{
  \fill (\ix 6pt, 0pt) circle[radius=2pt];
}
\draw (15 6pt, 0pt) circle[radius=4pt] node[above=3pt] {$k$};

\draw[-] (-1 6pt, 0pt) -- (17 6pt, 0pt);

\draw[dashed,thick=3pt] (4 6pt, -10pt) -- (4 6pt, 10pt)
  node[above] {$2d-k$};

\draw[-,thick=3pt] (12 6pt, -10pt) -- (12 6pt, 10pt)
  node[above] {$k-2d$};
\draw[->,thick=3pt] (12 6pt - 4pt, 7pt) -- (12 6pt + 4pt, 7pt);
\end{tikzpicture}%
\vspace{1.5em}

\begin{tikzpicture}
\draw[draw=white] (-4.5,0.3) rectangle (-1,-0.3) node[pos=0.5] {$\big(\psi_k,k\big)$};

\foreach \ix in {0,...,4}{
  \draw (\ix 6pt, 0pt) circle[radius=2pt];
}
\foreach \ix in {5,...,12}{
  \fill (\ix 6pt, 0pt) circle[radius=2pt];
}
\foreach \ix in {13,...,16}{
  \fill (\ix 6pt, 0pt) circle[radius=2pt];
}
\draw (11 6pt, 0pt) circle[radius=4pt] node [above=3pt] {$k\hspace{1em}$};

\draw[-] (-1 6pt, 0pt) -- (17 6pt, 0pt);

\draw[dashed,thick=3pt] (4 6pt, -10pt) -- (4 6pt, 10pt)
  node[above] {$-k-2\hspace{1em}$};

\draw[-,thick=3pt] (12 6pt, -10pt) -- (12 6pt, 10pt)
  node[above] {$\hspace{1em}k+2$};
\draw[->,thick=3pt] (12 6pt - 4pt, 7pt) -- (12 6pt + 4pt, 7pt);
\end{tikzpicture}%
\vspace{1.5em}

\begin{tikzpicture}
\draw[draw=white] (-4.5,0.3) rectangle (-1,-0.3) node[pos=0.5] {$\big(\wtd\phi_k, k\big)$};

\foreach \ix in {0,...,4}{
  \fill (\ix 6pt, 0pt) circle[radius=2pt];
}
\foreach \ix in {5,...,12}{
  \draw (\ix 6pt, 0pt) circle[radius=2pt];
}
\foreach \ix in {13,...,16}{
  \draw (\ix 6pt, 0pt) circle[radius=2pt];
}
\draw (4 6pt, 0pt) circle[radius=4pt];

\draw[-] (-1 6pt, 0pt) -- (17 6pt, 0pt);

\draw[-,thick=3pt] (4 6pt, -10pt) -- (4 6pt, 10pt)
  node[above] {$k$};
\draw[->,thick=3pt] (4 6pt + 4pt, 7pt) -- (4 6pt - 4pt, 7pt);

\draw[dashed,thick=3pt] (12 6pt, -10pt) -- (12 6pt, 10pt)
  node[above] {$-k$};
\end{tikzpicture}%
\vspace{1.5em}

\begin{tikzpicture}
\draw[draw=white] (-4.5,0.3) rectangle (-1,-0.3) node[pos=0.5] {$\big(\wtd\psi_k, k\big)$};

\foreach \ix in {0,...,4}{
  \fill (\ix 6pt, 0pt) circle[radius=2pt];
}
\foreach \ix in {5,...,11}{
  \fill (\ix 6pt, 0pt) circle[radius=2pt];
}
\foreach \ix in {12,...,16}{
  \draw (\ix 6pt, 0pt) circle[radius=2pt];
}
\draw (5 6pt, 0pt) circle[radius=4pt] node[above=3pt] {$\hspace{1em}k$};

\draw[-] (-1 6pt, 0pt) -- (17 6pt, 0pt);

\draw[-,thick=3pt] (4 6pt, -10pt) -- (4 6pt, 10pt)
  node[above] {$k-2\hspace{1em}$};
\draw[->,thick=3pt] (4 6pt + 4pt, 7pt) -- (4 6pt - 4pt, 7pt);

\draw[dashed,thick=3pt] (12 6pt, -10pt) -- (12 6pt, 10pt)
  node[above] {$\hspace{1em}-k+2$};
\end{tikzpicture}%
\end{center}

We now come back to the spectral decomposition of Poincar\'e series. We illustrate our approach in the first case $\rmP_{k[d],\ell;n,m}$. It is central for the more general discussion in Section~\ref{sec:harish-chandra-modules} to decompose the action of lowering operators and raising operators into two parts according to the following analogue to the Leibniz product rule.
\begin{multline*}
  \rmL_{k+\ell} \Big(
  \phi_{k[d]}(n;\tau)
  \cdot
  \phi_\ell(m;\tau) \big|_{k+\ell} \,\gamma
  \Big)
=
  \rmL_{k+\ell} \Big(
  \phi_{k[d]}(n;\tau) \big|_k \,\gamma
  \cdot
  \phi_\ell(m;\tau) \big|_\ell \,\gamma
  \Big)
\\
=
  \rmL_k \Big(
  \phi_{k[d]}(n;\tau) \big|_k \,\gamma \
  \Big)
  \cdot
  \phi_\ell(m;\tau) \big|_\ell \,\gamma
  \,+\,
  \phi_{k[d]}(n;\tau) \big|_k \,\gamma
  \cdot
  \rmL_\ell \Big(
  \phi_\ell(m;\tau) \big|_\ell \,\gamma
  \Big)
\tx{.}
\end{multline*}
An analogous formula holds for the raising operator. Applying both formulas iteratively yields that 
\begin{multline*}
  \CC[\rmL, \rmR]\, \big( \rmP_{k[d],\ell;n,m} ,\, k+\ell \big)
\subseteq
  \CC[\rmL, \rmR]\, \big( \phi_{k[d]}(n;\,\cdot\,) \cdot \phi_\ell(m;\,\cdot\,) ,\, k+\ell \big)
\\
\subseteq
  \Big( \CC[\rmL, \rmR]\, \big( \phi_{k[d]}(n;\,\cdot\,),\, k \big) \Big)
  \otimes
  \Big( \CC[\rmL, \rmR]\, \big( \phi_\ell(m;\,\cdot\,),\, \ell \big) \Big)
\tx{.}
\end{multline*}
The tensor product of $(f,k)$ and $(g,\ell)$ is defined as $(fg, k+\ell)$, which is in accordance with the product $f|_k\,\gamma \cdot g|_\ell\,\gamma = (fg)|_{k+\ell}\,\gamma$.

With this machinery, we can now read off an upper bound for the support of the graded module $\CC[\rmL, \rmR]\, \big(\rmP_{k[d],\ell;n,m},\, k + \ell \big)$. We have already given diagrams for the left and right tensor component that correspond to it. All nonzero graded components are $1$-dimensional, and we find that the weights of the nonzero graded components in the tensor product are at least $k-2d+\ell$.

As a last step, we use the classification of Harish-Chandra modules for $\SL{2}(\RR)$: (Limits of) holomorphic discrete series are the only Harish-Chandra admissible modules which have a lowest weight. Therefore, the Harish-Chandra module attached to $\rmP_{k[d],\ell;n,m}$ is a (finite) direct sum of (limits of) holomorphic discrete series. Graded components of such discrete series correspond to almost holomorphic modular forms, and this proves again what we have already observed in Section~\ref{ssec:elliptic:almost-holomorphic}: The Poincar\'e series $\rmP_{k[d],\ell;n,m}$ decomposes as a finite sum of almost holomorphic modular forms.

\paragraph{The inconclusive cases}
To analyze the Poincar\'e series in~\eqref{eq:elliptic:harmonic-poincare-series} and~\eqref{eq:elliptic:antiholomorphic-poincare-series}, we have to consider the tensor product of the $\CC[\rmL,\rmR]$-modules generated by $\wtd\phi_k$ and $\phi_\ell$, and by $\wtd\psi_k$ and $\phi_\ell$, respectively. These tensor products are supported on all weights and all weight spaces are infinite dimensional. Thus, we cannot deduce anything definite. Nevertheless, this at least suggests that infinitely many Maass cusp forms appear in the spectral decomposition, which is what we verified directly in Section~\ref{ssec:elliptic:spectral-decomposition}.

\section{Real-analytic Siegel modular forms}
\label{sec:siegel}

We start by introducing necessary notation to define Siegel modular forms. Let $I = I^{(2)}$ be the~$2 \times 2$ identity matrix, $\Gamma := \Gamma^{(2)} := \Sp{2}(\ZZ)$ be the symplectic group of degree~$2$ over $\ZZ$, $\HS^{(2)}$ be the Siegel upper half space of degree~$2$, and let $Z = X+iY \in \HS^{(2)}$ be a typical variable. If $M = \begin{psmatrix} A & B \\ C & D \end{psmatrix} \in \Gamma$ and $Z \in \HS^{(2)}$, then
\begin{gather*}
  M \bullet Z
:=
  (AZ+B) (CZ+D)^{-1}
\tx{.}
\end{gather*}
Furthermore, if $F:\, \HS^{(2)} \ra \CC$ and if $k \in \ZZ$, then
\begin{gather}
\label{eq:siegel:slash-action}
  \big( F \big|_k\,M\ \big) (Z)
:=
  \det(CZ+D)^{-k}\, F(M \bullet Z)
\end{gather}
for all $M=\begin{psmatrix}A & B\\ C & D\end{psmatrix} \in \Gamma$.

\begin{definition}
A holomorphic (degree~$2$) Siegel modular form of weight $k$ on $\Gamma$ is a holomorphic function $F :\, \HS^{(2)} \ra \CC$ such that, for all $M \in \Gamma$, $F\big|_k\,M = F$.
\end{definition}

\subsection{Almost holomorphic Siegel modular forms}
\label{ssec:siegel:almost-holomorphic}

Almost holomorphic Siegel modular forms were introduced by Shimura~\cite{shimura-1987}. In degree~$2$, they were classified in~\cite{pitale-saha-schmidt-2015,klemm-poretschkin-schimannek-raum-2015}. Write $Z = (z_{i\!j})$ and set $\ppartial{\ov{Z}} := \big( \frac{1}{2}(1 + \delta_{i\!j}) \ppartial{\ov{z_{i\!j}}} \big)$ to define the lowering operator
\begin{gather}
\label{eq:siegel:lowering-operator}
  \rmL := \rmL^{(2)}
:=
  Y \rT\, \big( Y \ppartial{\ov{Z}} \big)
\tx{.}
\end{gather}
By the $d$\thdash\ power of $\rmL$, we mean its $d$\thdash\ tensor power.
\begin{definition}
An almost holomorphic (degree~$2$) Siegel modular form of weight $k$ and depth~$d$ on $\Gamma$ is a real-analytic function $F :\, \HS^{(2)} \ra \CC$ satisfying the following conditions:
\begin{enumeratearabic}
\item For all $M \in \Gamma$, $F\big|_k\,M = F$.

\item We have that $\rmL^{d+1}(F) = 0$.
\end{enumeratearabic}
\end{definition}

\subsection{Harmonic Siegel-Maass forms}
\label{sec:siegel:harmonic-maass}

In \cite{bringmann-raum-richter-2011}, we introduced a certain space of harmonic (skew) Siegel-Maass forms, and we proved a connection of this space to the space of harmonic skew-Maass-Jacobi forms.  In particular, we answered a question of Kohnen~\cite{kohnen-1993} on how skew-holomorphic Jacobi forms are related to real-analytic Siegel modular forms.  We now introduce the $(\alpha,\beta)$\nbd slash action and the corresponding matrix-valued Laplace operator to recall the definition of harmonic (skew) Siegel-Maass forms in~\cite{bringmann-raum-richter-2011}.
 
If $F :\, \HS^{(2)} \ra \CC$, and if $\alpha, \beta \in \CC$ such that $\alpha-\beta\in\ZZ$, then
\begin{gather}
\label{eq:siegel:slash-action-general}
  \big( F\,\big|_{(\alpha,\,\beta)}\,M \big)(Z)
:=
  \det(CZ+D)^{-\alpha}\, \det(C\ov{Z} + D)^{-\beta}\, F(M \bullet Z)
\end{gather}
for all $M=\begin{psmatrix}A & B\\ C & D\end{psmatrix} \in \Gamma$. If $\beta=0$, we have $\big( F \big|_\alpha\, M  \big)(Z) = \big( F \big|_{(\alpha,\,\beta)}\,M \big)(Z)$.  As before, write $Z=(z_{i\!j})$ and set $\ppartial{Z} := \big( \frac{1}{2}(1 + \delta_{i\!j}) \ppartial{z_{i\!j}} \big)$ to define the Laplace operator
\begin{gather}
\label{eq:siegel:laplace}
  \Omega_{\alpha, \beta}
:=
  - 4 Y \rT\, \big( Y \ppartial{\ov{Z}} \big) \ppartial{Z}
  - 2 i \beta Y \ppartial{Z} + 2 i \alpha Y \ppartial{\ov{Z}}
\tx{,}
\end{gather}
which is equivariant with respect to the action in \eqref{eq:siegel:slash-action-general} (see~\cite{Maass-1971} for details). For the remainder, assume that $\kappa$ is an odd integer such that $\kappa \not\in \{1,3\}$.

\begin{definition}
\label{def:siegel:harmonic-siegel-maass}
A harmonic (skew) Siegel-Maass form of weight $\kappa$ on $\Gamma$ is a real-analytic function $F :\, \HS^{(2)} \ra \CC$ satisfying the following conditions:
\begin{enumeratearabic}
\item For all $M\in\Gamma$, $F\,|_{\big( \frac{1}{2},\,\kappa-\frac{1}{2} \big)}\, M = F$.

\item We have that $\Omega_{\frac{1}{2}, \kappa-\frac{1}{2}}(F) = 0$.

\item We have that $|F(Z)| \leq c \, \tr(Y)^a$ for some $a,c>0$ as $\tr(Y) \ra \infty$.
\end{enumeratearabic}

Let $\bbM^\sk_\kappa$ denote the space of such harmonic Siegel-Maass forms of weight~$\kappa$.
\end{definition}

\begin{remark}
\label{rem:skew_vs_hol}
In~\cite{bringmann-raum-richter-2011} we only focused on (skew) Siegel-Maass forms of type $\big( \frac{1}{2}, \kappa-\frac{1}{2} \big)$, in order to establish links to the spaces of skew-Maass Jacobi forms (if $\kappa<0$) and skew-holomorphic Jacobi forms (if $\kappa>3$).  Nevertheless, many results of~\cite{bringmann-raum-richter-2011} extend to ``holomorphic weights'', which are more natural from a representation theoretic perspective. In fact, if $F \in \bbM^\sk_\kappa$, then the form $\det (Y)^{\kappa - 1 \slash 2} F(Z)$ has weight $k:=1-\kappa$, i.e., it is  invariant under \eqref{eq:siegel:slash-action} with $k = 1 - \kappa$, where $k$ is an even integer such that $k \not\in \{0, -2\}$. For convenience, we set
\begin{gather}
  \bbM_k
:=
  \bbM^{(2)}_k
:=
  \det(Y)^{\frac{1}{2} - k}\, \bbM^\sk_{1 -k}
\tx{.}
\end{gather}
\end{remark}

Let $\Gamma_{\infty} := \Gamma^{(2)}_{\infty} := \big\{\begin{psmatrix} A & B \\ 0 & D \end{psmatrix} \in \Gamma\, \big\}$. Recall Maass'~\cite{Maass-1953, Maass-1971} nonholomorphic Eisenstein series
\begin{gather}
\label{eq:maass-eisenstein-series}
  E_{\alpha, \beta}(Z)
:=
  \sum_{M \in \Gamma_{\infty}\setminus \Gamma}
  1 \big|_{(\alpha, \beta)}\, M
\end{gather}
and also the Poincaré-Eisenstein series
\begin{gather}
\label{eq:poincare-eisenstein-series}
  P_{\kappa,s}(Z)
:=
  \sum_{M \in \Gamma_{\infty} \backslash \Gamma}
  \det (Y)^s \big|_{\left(\frac{1}{2},\,\kappa -\frac{1}{2}\right)}\, M
\tx{.}
\end{gather}

Then $P_{\kappa,s} = \det(Y)^s E_{s + 1 \slash 2, s + \kappa - 1 \slash 2}$, and in~\cite{bringmann-raum-richter-2011} we stated the following fact.
\begin{proposition}
\label{prop:harmonic-siegel-maass-examples}
If $s=0$ ($\kappa>3$) or $s=\frac{3}{2} -\kappa$ ($\kappa<0$), then $P_{\kappa,s} \in \bbM^\sk_\kappa$. In other words, we have
\begin{alignat*}{3}
&
  \det(Y)^{\frac{1}{2} - k}\, P_{1-k, 0}
&{}\in
  \bbM_k
\quad&
  \tx{for $k < -2$;}
\\
&
  \det(Y)^{\frac{1}{2} - k}\, P_{1-k, \frac{1}{2} + k}
&{}\in
  \bbM_k
\quad&
  \tx{for $k > 1$.}
\end{alignat*}
\end{proposition}

We end this section with a remark on terms of the Fourier series of harmonic Siegel-Maass forms.
\begin{remark}
\label{rm:siegel:fourier-expansion}
In~\cite{bringmann-raum-richter-2011}, we determined Fourier series expansions of harmonic Siegel-Maass forms: For any non-degenerate, symmetric, half-integral $2 \times 2$ matrix~$T$ there exist functions $\Psi_k(T;\, Z)$ such that the $T$\thdash\ term of the Fourier series of any $F \in \bbM_k$ is given by
\begin{gather*}
  c(F;\, T)\, \Psi_k(T;\,Z)
\tx{,}
\qquad
  \tx{where $c(F;\, T) \in \CC$.}
\end{gather*}
The ideas in the proof of Proposition~\ref{prop:siegel-fourier-coefficient-is-not-almost-holomorphic} can be used to show that for $k < -2$, $\Psi_k(T;Z) = 0$ for positive definite $T$.
\end{remark}

\section{Real-analytic Siegel-Poincar\'e series}
\label{sec:poincare-series}

In this Section, we define Poincar\'e series attached to products of terms of Fourier series of holomorphic Siegel modular forms and harmonic Siegel-Maass forms. Recall from Proposition~\ref{prop:harmonic-siegel-maass-examples} that $\det(Y) E_{1+k,1} \in \bbM_k$. We can use its Fourier series to define $\Psi_k(T;Z)$ in Remark~\ref{rm:siegel:fourier-expansion}.

If $\Re(\alpha)$, $\Re(\beta)>\frac{1}{2}$, then (see~\cite{Maass-1971}) the terms of the Fourier series of the Eisenstein series $E_{\alpha, \beta}$ are, up to scalar multiples, given by
\begin{gather*}
  h_{\alpha, \beta}(T;Y)\, e^{2\pi i\, \tr(TX)}
\tx{,}
\end{gather*}
where
\begin{gather}
\label{eq:def_of_h}
  h_{\alpha, \beta}(T;Y)
:=
  \mathop{\int\!\!\int}_{U\pm T>0}
  \det(U+T)^{\alpha-\frac{3}{2}}\, \det(U-T)^{\beta-\frac{3}{2}}
  e^{-2\pi \tr(YU)} \;
  d\!U
\tx{.}
\end{gather}

For the remainder, let $T$ be positive definite, which we denote by~$T > 0$. Set
\begin{gather}
\label{eq:poincare:harmonic-fourier-coefficient}
  \Psi_k(T;Z)
:=
  \det (T Y)\, h_{k+1,1}(T;\, Y)\, e^{2 \pi i\, \tr(T X)}
\quad\tx{and}\quad
  \Phi_\ell(T;Z)
:=
  e^{2 \pi i\, \tr(TZ)}
\,\tx{,}
\end{gather}
and for $T' > 0$ and positive even integers~$k$ and $\ell$ define the Poincar\'e series
\begin{gather}
\label{eq:siegel:poincare-series}
  P^{(2)}_{k,\ell;\, T, T'}(Z)
:=
  \sum_{M \in \Delta \backslash \Gamma}
  \big( \Psi_k(T;Z) \cdot \Phi_\ell(T';Z) \big) \big|_{k+\ell}\, M
\tx{,}
\end{gather}
where $\Delta$ is the subgroup of $\Gamma$ defined by
\begin{gather*}
  \Delta
:=
  \left\{ \begin{psmatrix} I & B \\ 0 & I \end{psmatrix}, B \in \Mat{2}(\ZZ), \rT B = B \right\}
\tx{.}
\end{gather*}
Observe that $\Phi_k (T; Z+B) = \Phi_k (T; Z)$ and $\Psi_\ell(T';Z+B)=\Psi_\ell(T';Z)$, and one finds that $ P^{(2)}_{k,\ell;\, T, T'}$ is well-defined.

\subsection{Convergence}  

We determine the convergence of $ P^{(2)}_{k,\ell;\, T, T'}(Z)$ by comparing it to the  Poincar\'e series
\begin{gather*}
  P_{k',T}(Z)
:=
  \sum_{M \in \Delta\setminus\Gamma}
  e^{2 \pi i\, \tr(TZ)} \big|_{k'}\, M
\text{,}
\end{gather*}
which converges absolutely and uniformly on compact subsets of $\HS^{(2)}$ for even integers $k' \ge 6$ (see, for example, Proposition~3 on page~85 of~\cite{klingen-1990}). 

Consider $\det(Y)\, h_{k+1,1}(T;\, Y)$, where  $T > 0$ and $Y>0$. Let $0<\lambda_1\leq \lambda_2$ be the eigenvalues of $TY$.  Then $\det(TY)=\lambda_1\lambda_2$, $\tr(TY)=\lambda_1+\lambda_2$, and $\lambda_1\geq\frac{\det(TY)}{\tr(TY)}$.

Shimura~\cite{shimura-1982} studied a function $\omega(g,h;\alpha,\beta)$, which is closely related to \eqref{eq:def_of_h}.  Specifically, if $g=2\pi Y$, $h=T$, $\alpha=k+1$, and $\beta=1$, then
\begin{gather*}
  \omega(g,h;\alpha,\beta)
:=
  \frac{2^{-2-2k}}{\pi}\,
  \det(2\pi TY)^{\frac{1}{2}-k} \det(2\pi Y)^{k+\frac{1}{2}}
  \cdot
   h_{1+k,1}(T;\, Y)
\text{.}
\end{gather*} 
Shimura established an estimate for $\omega$, which implies that there exist constants $a,b>0$ (depending on $\alpha$ and $\beta$, i.e., on $k$) such that
\begin{gather}
\label{eq:shimuras-A-and-B}
\begin{split}
  \det(T Y) h_{k+1,1}(T;\, Y)
&{}
\leq
  \det(T Y)\, 2^{2k+2}\pi\,
  \det(2\pi TY)^{k - 1 \slash 2}
  \det(2\pi Y)^{-k - 1 \slash 2}\;
   a\, e^{-2 \pi\, \tr(TY)} \big(1 + \lambda_1^{-b} \big)
\\
&{}
\leq
  c\, e^{-2\pi\, \tr(TY)}
  \Big( 1 + \Big( \frac{\tr(TY)}{\det(TY)} \Big)^b \Big)
\tx{,}
\end{split}
\end{gather} 
where $c:=a \, 2^{1+2k} \det(T)^{k - b + 1 \slash 2}$. Note that $\tr(TY)^b\, e^{-2\pi\, \tr(TY)} \leq d e^{-\pi\, \tr(TY)}$ for some constant $d>0$. Thus, 
\begin{gather}
\label{eq:bound_for_h}
  \det(T Y) h_{k+1,1}(T;Y)
\ll
  \big( 1 + \det(Y)^{-b} \big)\, e^{-\pi\, \tr(TY)}
\tx{,}
\end{gather}
and one finds that $ P^{(2)}_{k,\ell;\, T, T'}(Z)$ in~\eqref{eq:siegel:poincare-series} is dominated by
\begin{gather}
\label{eq:majorant}
  \sum_{M \in \Delta \backslash \Gamma}
  \Big( \big( 1 + \det(Y)^{-b} \big)\, e^{-\pi\, \tr(\wtd{T}Y)} \Big)
  \big|_{k+\ell}\, M
\end{gather}
with $\wtd{T} := T+T' > 0$. We conclude that $P^{(2)}_{k,\ell;\, T, T'}(Z)$ converges absolutely and uniformly on compact subsets of $\HS^{(2)}$ if $\ell+k-2b\ge6$.

\subsection{Non-vanishing}

In this Section, we show that $\lim_{\ell \ra \infty} P^{(2)}_{k,\ell;\, T, T'}(iy_0 I)>0$ for some $y_0>1$, which implies that $\bbP^{(2)}_{k,\ell;T,T'}$ does not vanish identically for all $\ell$ large enough.  We apply the following lemma of Kowalski, Saha, and Tsimerman~\cite{kowalski-saha-tsimerman-2011}, 
where $U(y_0)$ stands for some neighborhood of $i y_0 I$ ($y_0>0$). 

\begin{lemma}[{\cite{kowalski-saha-tsimerman-2011}}]
\label{lem: KST}
There exists a real number $y_0 > 1$ such that for any $\left(\begin{smallmatrix} A & B \\ C & D \end{smallmatrix}\right) \in \Gamma$ with $C \neq 0$ and for all $Z \in U(y_0)$, we have $|\det(CZ+D)|>1$.
\end{lemma}

Let $M=\left(\begin{smallmatrix} A & B \\ C & D \end{smallmatrix}\right) \in \Gamma$. Suppose that $C\not=0$. Choose $y_0>1$ as in Lemma~\ref{lem: KST} and consider $Z=i y_0 I$.  Then $|\det(CZ+D)|>1$, and if $T > 0$, then
\begin{gather*}
  \det(CZ+D)^{-k-\ell}
  \big(1+|\det(CZ+D)|^{2b}\big)\,
  e^{2 \pi i \tr(T\, M \bullet Z)}
\xrightarrow{\ell \ra \infty}
  0
\tx{.}
\end{gather*}
Recall that~\eqref{eq:majorant} is a majorant of~\eqref{eq:siegel:poincare-series}, which converges for all $\ell \ge 6-k+2b$, and we find that
\begin{gather*}
  \sum_{\substack{M = \begin{psmatrix}\ast & \ast \\ C & \ast \end{psmatrix}
                      \in \Delta \backslash \Gamma \\
                  C \ne 0}}
  \Big( \big(\Psi_k(T;Z) \cdot \Phi_\ell(T';Z)\big) \Big)
  \big|_{k+\ell}\, M
\xrightarrow{\ell \ra \infty}
  0
\tx{.}
\end{gather*}

Finally, suppose that $C=0$. Then $D = \rT{A}^{-1}$. Moreover, we factor out by $\Delta$, and hence we may (and do) assume that $B=0$. 
If $M=\left(\begin{smallmatrix} A & 0 \\ 0 & \rT A^{-1} \end{smallmatrix}\right)$ and again $Z=i y_0 I$, then
\begin{gather*}
  \det\big( \Im\big(M \bullet i y_0 I \big)\big)
=
  \det\big( A\, (y_0 I) \rT{A} \big)
=
  y_0^2 \det(A)^2
=
  y_0^2
\tx{.}
\end{gather*}
Thus, 
\begin{gather*}
  \sum_{M = \begin{psmatrix}\ast & \ast \\ 0 & \ast \end{psmatrix}
            \in \Delta \backslash \Gamma}
  \Big( \big(\Psi_k(T;Z) \cdot \Phi_\ell(T';Z)\big) \Big)
  \big|_{k+\ell}\, M
=
  y_0^2
  \sum_{A\in \GL{2}(\ZZ)}
  \det(A)^{-k-\ell} h_{k+1,1}(T, y_0 I)
  e^{-2\pi y_0\, \tr(\wtd{T} A \rT{A})}
\,\tx{,}
\end{gather*}
which is positive ($k$ and $\ell$ are even), where again $\wtd{T} := T+T' > 0$. In particular,  
\begin{gather*}
  \lim_{\ell \ra \infty}
  P^{(2)}_{k,\ell; T, T'}(i y_0 I)
>
  0
\tx{.}
\end{gather*}

\section{\texpdf{$(\frakg,K)$-modules generated by Poincar\'e series}{(g,K)-modules generated by Poincar\'e series}}
\label{sec:harish-chandra-modules}

Throughout this section, we focus on \gKmodules\ for $G := \Sp{2}(\RR)$.  We diverge from the classical notation $M \in \Sp{2}(\RR)$ in favor of the representation theoretic notation $g \in G$. Recall the realization of the symplectic group as
\begin{gather*}
  \Sp{2}(\RR)
=
  \big\{ g \in \Mat{4}(\RR) \,:\, \rT g J^{(2)} g = J^{(2)} \big\}
\tx{,}
\quad
  J^{(2)} = \begin{pmatrix} 0 & -I^{(2)} \\ I^{(2)} & 0 \end{pmatrix}
\tx{,}
\end{gather*}
with maximal compact subgroup
\begin{gather*}
  K
=
  \Big\{
  \begin{pmatrix} A & -B \\ B & A \end{pmatrix} \,:\,
  A + iB \in \U{2}(\RR)
  \Big\}
\cong
  \U{2}(\RR)
\tx{,}
\end{gather*}
where $\U{2}(\RR)$ are the $\RR$-points of the unitary group $\U{2}$ attached to the quadratic extension $\CC \slash \RR$, which as an algebraic group is defined over $\RR$.

Irreducible representations of $\U{2}(\RR)$ are isomorphic to $\det^k \sym^\ell := \det^k \otimes \sym^\ell$ for some $k \in \ZZ$, $\ell \in \ZZ_{\ge 0}$. This is the classical way of denoting weights for Siegel modular forms. In the context of real analytic representation theory, it is more common to parametrize irreducible $K$\nbd representations by integers $a,b \in \ZZ$ subject to the condition that $a \ge b$. This notion stems from the action of the center of $\frakk$ (the complexified Lie algebra of~$K$) on a representation. Translation between the two conventions is straightforward: The pair $(a,b)$ corresponds to $\det^b \sym^{b-a}$, while the weight $\det^k \sym^\ell$ corresponds to $(k+\ell,k)$.

A \gKmodule\ is a simultaneous $\frakg$ and $K$-module with compatibility relations imposed on them. A precise definition can be found in Section 3.3.1 of~\cite{wallach-1988}. One invariant of \gKmodules\ is the set of non-trivial $K$\nbd types (i.e.,\ irreducible $K$\nbd representations). A \gKmodule~$\varpi$ viewed as a $K$\nbd representation can be decomposed as a direct sum of irreducibles. We say that a $K$\nbd type $\pi_K$ occurs in $\varpi$ if $\dim\, \Hom_K(\pi_K, \varpi) > 0$. If the multiplicity in $\varpi$ of all $\pi_K$'s is finite, then $\varpi$ is called a Harish-Chandra module. The invariant that we will primarily encounter is the set of $K$\nbd types that occur in a \gKmodule.

The $K$\nbd types in \gKmodules\ for $\Sp{2}(\RR)$ can be displayed by a half-grid in 2 dimensions. A typical such grid looks as follows, where we have marked $K$\nbd types that occur by a filled circle, and those that do not occur by an empty circle.
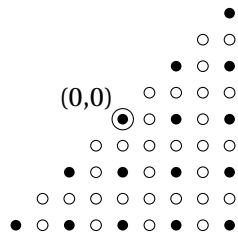
\begin{center} 
\begin{tikzpicture}
\draw (40pt,40pt) circle[radius=4pt] node[above left] {(0,0)};                                                                                       

\fill (0pt,0pt) circle[radius=2pt];
\draw (10pt,0pt) circle[radius=2pt]; 
\draw (10pt,10pt) circle[radius=2pt];                                                                                                                
\fill (20pt,0pt) circle[radius=2pt]; 
\draw (20pt,10pt) circle[radius=2pt];
\fill (20pt,20pt) circle[radius=2pt];
\draw (30pt,0pt) circle[radius=2pt]; 
\draw (30pt,10pt) circle[radius=2pt];
\draw (30pt,20pt) circle[radius=2pt];
\draw (30pt,30pt) circle[radius=2pt];
\fill (40pt,0pt) circle[radius=2pt]; 
\draw (40pt,10pt) circle[radius=2pt];
\fill (40pt,20pt) circle[radius=2pt];
\draw (40pt,30pt) circle[radius=2pt];
\fill (40pt,40pt) circle[radius=2pt];
\draw (50pt,0pt) circle[radius=2pt]; 
\draw (50pt,10pt) circle[radius=2pt];
\draw (50pt,20pt) circle[radius=2pt];
\draw (50pt,30pt) circle[radius=2pt];
\draw (50pt,40pt) circle[radius=2pt];
\draw (50pt,50pt) circle[radius=2pt];
\fill (60pt,0pt) circle[radius=2pt]; 
\draw (60pt,10pt) circle[radius=2pt];
\fill (60pt,20pt) circle[radius=2pt];
\draw (60pt,30pt) circle[radius=2pt];
\fill (60pt,40pt) circle[radius=2pt];
\draw (60pt,50pt) circle[radius=2pt];
\fill (60pt,60pt) circle[radius=2pt];
\draw (70pt,0pt) circle[radius=2pt]; 
\draw (70pt,10pt) circle[radius=2pt];
\draw (70pt,20pt) circle[radius=2pt];
\draw (70pt,30pt) circle[radius=2pt];
\draw (70pt,40pt) circle[radius=2pt];
\draw (70pt,50pt) circle[radius=2pt];
\draw (70pt,60pt) circle[radius=2pt];
\draw (70pt,70pt) circle[radius=2pt];
\fill (80pt,0pt) circle[radius=2pt]; 
\draw (80pt,10pt) circle[radius=2pt];
\fill (80pt,20pt) circle[radius=2pt];
\draw (80pt,30pt) circle[radius=2pt];
\fill (80pt,40pt) circle[radius=2pt];
\draw (80pt,50pt) circle[radius=2pt];
\fill (80pt,60pt) circle[radius=2pt];
\draw (80pt,70pt) circle[radius=2pt];
\fill (80pt,80pt) circle[radius=2pt];
\end{tikzpicture}%
\end{center}

We say that a Harish-Chandra module~$\varpi$ has a vertical wall in the direction of~$\leftarrow$ or $\rightarrow$, if there exists $a_0 \in \ZZ$ such that every $K$-type with highest weight $(a,b)$ that occurs in $\varpi$ satisfies $a \le a_0$ or $a \ge a_0$, respectively. Horizontal walls in the direction of~$\uparrow$ and $\downarrow$ can be defined analogously.

\subsection{\texpdf{$(\frakg,K)$-modules associated to modular forms}{(g,K)-modules associated to modular forms}}
\label{ssec:modular-forms-harish-gK-modules}

Section~2 of~\cite{raum-2015a} gives an account of the connection between modular forms and Harish-Chandra modules. Weights are finite dimensional, holomorphic representations~$\sigma$ of $\GL{2}(\CC)$. Their representation space is denoted by~$V(\sigma)$. Given a weight~$\sigma$ and a smooth function~$F :\, \HS^{(2)} \ra V(\sigma)$, we can attach a function~$\rmA_{\RR,\sigma}(F) := \rmA_\RR(F)$ on $\rmG(\RR) = \Sp{2}(\RR)$:
\begin{gather}
  \rmA_\RR(F)(g)
=
  \sigma^{-1}\big( j(g, iI^{(2)}) \big)\, f(g \bullet i I^{(2)})
\tx{,}\qquad
  j(g, Z)
=
  C Z + D
\tx{.}
\end{gather}
From $\rmA_\RR(F)$ one constructs the vector space~$\ov\rmA_{\RR,\sigma}(F)(g)$ that is spanned by its coordinates. This space under right translation by $K$ is isomorphic to the dual~$\sigma^\vee$ of $\sigma$. The action of $\frakg$ on this space generates a \gKmodule\ that we denote by $\varpi(F)$. If $F$ is an automorphic form or a term of the Fourier series of a modular form, then $\varpi(F)$ is a Harish-Chandra module.

Recall the compatibility of covariant differential operators acting on~$F$ and the $\frakg$\nbd action on $\varpi(F)$ that is stated and the end of Section~2.2 of~\cite{raum-2015a} in terms of the following two commutative diagrams.
\begin{center}
\begin{tikzpicture}
\matrix(m)[matrix of math nodes,
column sep = 10em, row sep = 3em,
text height = 1.5em, text depth = 1.25ex]
{ F & \ov{\rmA}_{\RR,\sigma}(F) \\
  \rmR_\sigma(F) & \frakm^+\,\ov{\rmA}_{\RR,\sigma}(F) \\
};

\path
(m-1-1) edge[|->] node[above] {$\rmA_{\RR,\sigma}$} (m-1-2)
(m-2-1) edge[|->] node[above] {$\rmA_{\RR,\sym^2\sigma}$} (m-2-2)

(m-1-1) edge[|->] node[left] {$\rmR_\sigma$} (m-2-1)
(m-1-2) edge[|->] node[right] {$\frakm^+$} (m-2-2);
\end{tikzpicture}%
\hspace*{4em}
\begin{tikzpicture}
\matrix(m)[matrix of math nodes,
column sep = 10em, row sep = 3em,
text height = 1.5em, text depth = 1.25ex]
{ F & \ov{\rmA}_{\RR,\sigma}(F) \\
  \rmL_\sigma(F) & \frakm^-\,\ov{\rmA}_{\RR,\sigma}(F) \\
};

\path
(m-1-1) edge[|->] node[above] {$\rmA_{\RR,\sigma}$} (m-1-2)
(m-2-1) edge[|->] node[above] {$\rmA_{\RR,\det^{-2}\sym^2\sigma}$} (m-2-2)

(m-1-1) edge[|->] node[left] {$\rmL_\sigma$} (m-2-1)
(m-1-2) edge[|->] node[right] {$\frakm^-$} (m-2-2);
\end{tikzpicture}
\end{center}
This compatibility allows us to pass back an forth between the classical description of covariant differential operators acting on modular forms and the representation theoretic perspective.

\begin{lemma}
\label{la:poincare-series-harish-chandra-module}
Let $c :\, \HS^{(2)} \ra V(\sigma)$ be smooth function such that the Poincar\'e series
\begin{gather*}
  P_c
=
  \sum_{\gamma \in \Gamma_\infty \backslash \Gamma}
  c \big|_\sigma\, \gamma
\end{gather*}
is locally absolutely convergent. Then there is an inclusion $\varpi(P_c) \hra \varpi(c)$.
\end{lemma}
\begin{proof}
This is an immediate consequence of viewing Poincar\'e series as intertwining maps from a suitable principal series to the automorphic spectrum. It can also be seen directly, by checking that
\begin{gather*}
  \rmA_\RR\big( P_c \big)
=
  \sum_{\gamma \in \Gamma_\infty \backslash \Gamma}
  \rmA_\RR(c) \circ \gamma
\tx{.}
\end{gather*}
\end{proof}

\begin{lemma}
\label{la:product-harish-chandra-module}
Given two smooth functions $c_1 :\, \HS^{(2)} \ra V(\sigma_1)$ and $c_2 :\, \HS^{(2)} \ra V(\sigma_2)$, then there is an inclusion $\varpi(c_1 \cdot c_2) \hra \varpi(c_1) \otimes \varpi(c_2)$.
\end{lemma}
\begin{proof}
This is a rephrasing of the Leibniz rule for differentials.
\end{proof}

\subsection{\texpdf{The tensor products of $(\frakg,K)$-modules}{The tensor products of (g,K)-modules}}

\begin{proposition}
\label{prop:tensor-product-with-walls}
Let $\varpi_1$ and $\varpi_2$ be Harish-Chandra modules. If $\varpi_1$ has a vertical wall in the direction of~$\ra$ and $\varpi_2$ has a horizontal wall in the direction of~$\uparrow$, then the tensor product~$\varpi_1 \otimes \varpi_2$ has a vertical wall in the direction of~$\ra$. If $\varpi_1$ has a vertical wall in the direction of~$\leftarrow$ and $\varpi_2$ has a horizontal wall in the direction of~$\downarrow$, then the tensor product~$\varpi_1 \otimes \varpi_2$ has a horizontal wall in the direction of~$\downarrow$.
\end{proposition}
\begin{proof}
We prove the first case and leave the second one to the reader. Let $a_0$ and $b'_0$ be such that $a \ge a_0$ if $(a,b)$ occurs in $\varpi_1$ and $b' \ge b'_0$ if $(a',b')$ occurs in $\varpi_2$. Let $(a,b)$ and $(a',b')$ be arbitrary $K$\nbd types in $\varpi_1$ and $\varpi_2$. Then by the Clebsch-Gordan rules, their tensor product contains $K$\nbd types of weight~$(a'',b'')$ with $a'' + b'' = a + a' + b + b'$ and $\max(a-b,a'-b') - \min(a-b,a'-b') \le a''-b'' \le a-b + a'-b'$.

Adding these two, we find that
\begin{gather*}
  2 a''
\ge 
  a + a' + b' + b'
  +
  \max(a-b,a'-b') - \min(a-b,a'-b')
\tx{.}
\end{gather*}
If $a-b \ge a' - b'$, then this equals $2 (a + b') \ge 2 (a_0 + b'_0)$. Otherwise, it equals $2 (a' + b)$ which is greater than $2 (a + b') \ge 2(a_0 + b'_0)$, because $a'-b' \ge a - b$.
\end{proof}

\subsection{Harish-Chandra modules with walls}

\begin{proposition}
\label{prop:irreducible-unitary-walls-scalar-K-type}
Assume the generalized Ramanujan conjecture for $\GL{4}$. Let $\varpi$ be an irreducible, cuspidal, automorphic representation, with Harish-Chandra module $\varpi_\infty$ at the infinite place. If $\varpi_\infty$ has a vertical or horizontal wall, and if $\varpi_\infty$ contains a scalar $K$\nbd type, then $\varpi_\infty$ is a holomorphic or antiholomorphic (limit of) discrete series.
\end{proposition}
\begin{proof}
We first show that the Harish-Chandra parameters $(s_1,s_2)$ of $\varpi_\infty$ are integral. The Langlands classification~\cite{knapp-2001} exhausts irreducible Harish-Chandra modules as irreducible quotients of induced representations. We use~\cite{muic-2009} to determine their $K$-types. We adopt Mui\'c's notation. Equations~(9.3-9.5) of~\cite{muic-2009} allow us to focus on the induced representations
\begin{gather}
\label{eq:la:harish-chandra-modules-with-walls}
\begin{alignedat}{3}
&
  |\;|^p \sgn^p \times |\;|^t \sgn^t \rtimes 1
&&\,\tx{,}\qquad
  \delta\big( |\;|^s \sgn^\epsilon, k \big) \rtimes 1
&&\,\tx{,}\quad
  \zeta\big( |\;|^s \sgn^\epsilon, k \big) \rtimes 1
\\
&
  |\;|^s \sgn^\epsilon \rtimes X(p,\pm)
&&\,\tx{,}\quad
  |\;|^s \sgn^\epsilon \rtimes V_p
\end{alignedat}
\end{gather}
with non-integral Harish-Chandra parameter. The first representation is irreducible by Lemma~9.1 of~\cite{muic-2009}. The remaining ones are irreducible by Theorem~12.1 of~\cite{muic-2009}. Hence it suffices to check $K$\nbd types of induced representations that occur in~\eqref{eq:la:harish-chandra-modules-with-walls}, which was done in Section~6 of~\cite{muic-2009}. This shows that $s_1, s_2 \in \ZZ$.

Assume that $\varpi_\infty$ is tempered. Among the (limits of) discrete series, only the holomorphic ones contain scalar $K$\nbd types. To show that no other tempered representation can occur, observe that tempered representations that are not (limits of) discrete series are fully induced from discrete series attached to the Levi factor of a parabolic subgroup by~\cite{knapp-zuckerman-1976}. Section~9 of~\cite{muic-2009} lists the non-tempered constituents of inductions of discrete series, and this reduces us to the Harish-Chandra parameter $(0,0)$. By Corollary~5.2 of~\cite{muic-2009} the induce representation $1 \times 1 \rtimes 1$ is irreducible. Lemma~6.1 reveals that it has no walls. The principal series
\begin{gather*}
  \sgn \times \sgn^\epsilon \rtimes 1
\cong
  \sgn^\epsilon \times \sgn \rtimes 1
\cong
  \sgn^\epsilon \rtimes X(0,+)
  \,\oplus\,
  \sgn^\epsilon \rtimes X(0,-)
\end{gather*}
contains $K$\nbd types $(k',k')$ for odd $k'$ only.

Assume that $\varpi_\infty$ is non-tempered. Then by the generalized Ramanujan conjecture for $\GL{4}$ and by Arthur's endoscopic classification~\cite{arthur-2013} the only non-tempered contributions to the automorphic spectrum are lifts of
\begin{enumeratearabic*}
\item Soudry type,
\item Saito-Kurokawa type,
\item Howe-Piatetski-Shapiro type, or
\item one-dimensional type.
\end{enumeratearabic*}
For a detailed explanation see~\cite{arthur-2004}. Local components at the infinite places can be determined via the local Langlands correspondence for reductive groups over the reals, which was established in~\cite{langlands-1989}. Sections~1 and~2 of~\cite{schmidt-2016} summarize both results briefly. For Soudry type lifts we apply Lemma~6.1 of~\cite{muic-2009} to discover that no scalar $K$\nbd types occur. The holomorphic Saito-Kurokawa lift does contain scalar $K$\nbd types. At the infinite place it is a holomorphic discrete series. A Saito-Kurokawa lift with integral Harish-Chandra parameter that is not a holomorphic discrete series contains no scalar $K$\nbd type by Lemmas~6.1 and~9.2 of~\cite{muic-2009}. The Howe-Piatetski-Shapiro type corresponds to the Langlands quotient of the Borel subgroup. In the case of integral Harish-Chandra parameters, it contains no scalar $K$\nbd type by, again, Lemma~6.1 and~9.2 of~\cite{muic-2009}. This also applies to the one-dimensional type. This establishes the claim.
\end{proof}

\subsection{Proof of the main theorem}

Before we investigate Poincar\'e series, we recall two results about their Fourier series:
\begin{proposition}
\label{prop:vertical-values-in-fourier-coefficient-harish-chandra-modules}
Let $c :\, \HS^{(2)} \ra V(\sigma)$ be a holomorphic function. Then $\varpi(c)$ has a horizontal wall in the direction of~$\uparrow$. Moreover, $\varpi(\Psi_k(T;Z))$ has a vertical wall in the direction of~$\rightarrow$, where $\Psi_k(T;Z)$ is defined in~\eqref{eq:poincare:harmonic-fourier-coefficient}.
\end{proposition}
\begin{proof}
The first statement is classical and follows from the description of holomorphic discrete series. The second one is a direct consequence of Proposition~4.1 of~\cite{raum-2015a}, or alternatively can be extracted from~\cite{lee-1996}.
\end{proof}

\begin{proposition}
\label{prop:siegel-fourier-coefficient-is-not-almost-holomorphic}
There is no $d \in \ZZ_{\ge 0}$ such that
\begin{gather*}
   \rmL^d \big(
   \Psi_k(T; \,\cdot\,)
   \Phi_\ell(T'; \,\cdot\,)
   \big)
\end{gather*}
vanishes.
\end{proposition}
\begin{proof}
Since $\rmL(\Phi_\ell(T;\,\cdot\,) = 0$, we have
\begin{gather*}
  \rmL^d \big(
  \Psi_k(T; \,\cdot\,)
  \Phi_\ell(T'; \,\cdot\,)
  \big)
=
  \rmL^d \big(
  \Psi_k(T; \,\cdot\,)
  \big) \cdot
  \Phi_\ell(T'; \,\cdot\,)
\tx{.}
\end{gather*}
Using the equivariance
\begin{gather*}
  \Psi_k(\rT U T U; Z)
=
  \Psi_k(T; \rT U Z U)
\tx{,}\quad
  U \in \GL{2}(\RR)
\end{gather*}
we can focus on the case that $\Psi_k(T;\,\cdot\,)$ is a nonvanishing term of the Fourier series of an Eisenstein series. We then obtain an embedding of the Harish-Chandra module~$\varpi_k$ generated by that Eisenstein series into the generalized Whittaker model~$W_{k,T}$ associated with $\Psi_k(T;\,\cdot\,)$ (which, in fact, is a Bessel model). Using the decomposition series of principal series given in~\cite{lee-1996}, we find that the $K$\nbd types $(k+2a,k)$ occur in~$\varpi_k$ for all $a \ge 0$. Moreover, all $K$\nbd types in~$\varpi_k$ occur with multiplicity at most~$1$. Let $\pi$ be the projection of the weight (i.e.\ $\GL{2}(\CC)$\nbd representation) $\det^k (\sym^2)^d$ to $\det^k \sym^{2d}$. Our argument shows that
\begin{gather*}
   \pi \Big( \rmL^d \big(
   \Psi_k(T; \,\cdot\,)
   \big) \Big)
\end{gather*}
generates the $K$\nbd type $(k+2d,k)$ of the image of $\varpi_k$ in $W_{k,T}$. In particular, it does not vanish.
\end{proof}

\begin{proof}[Proof of Theorem~\ref{thm:maintheorem}]
By Section~\ref{sec:poincare-series}, the Poincar\'e series $P^{(2)}_{k,\ell;T,T'}$ is a real-analytic cusp form. We have to show that it vanishes under some tensor power of the lowering operator~$\rmL$. By the connection of modular forms and \gKmodules, elaborated on in Section~\ref{ssec:modular-forms-harish-gK-modules}, it suffices to show that $\varpi\big( P^{(2)}_{k,\ell;T,T'} \big)$ is a finite sum of holomorphic (limits of) discrete series.

Cuspidality implies that $\varpi\big( P^{(2)}_{k,\ell;T,T'} \big)$ is a direct sum of irreducibles. The theorem is proved if we show that any of its irreducible subquotients is a holomorphic (limit of) discrete series.

Lemma~\ref{la:poincare-series-harish-chandra-module} asserts that we can restrict on the codomain of
\begin{gather*}
  \varpi\big( P^{(2)}_{k,\ell;T,T'} \big)
\lhra
  \varpi\big(
  \Psi_k(T;Z) \cdot \Phi_\ell(T';Z)
  \big)
\tx{,}
\end{gather*}
and Lemma~\ref{la:product-harish-chandra-module} allows us to further reduce our considerations to
\begin{gather*}
  \varpi\big(
  \Psi_k(T;Z) \cdot \Phi_\ell(T';Z)
  \big)
\lhra
  \varpi_{k,\ell;T,T'}
=
  \varpi\big( \Psi_k(T;Z) \big)
  \otimes
  \varpi\big( \Phi_\ell(T';Z) \big)
\tx{,}
\end{gather*}
where $\Phi_\ell(T';Z)$ is viewed as a function from $\HS^{(2)}$ to $V(\det^\ell)$.

Proposition~\ref{prop:vertical-values-in-fourier-coefficient-harish-chandra-modules} guarantees that the first tensor factor has a horizontal wall, and the second one has a vertical wall. Consequently, we can apply Proposition~\ref{prop:tensor-product-with-walls}. It implies that $\varpi_{k,\ell;T,T'}$ has a vertical wall in the direction of~$\rightarrow$. If it occurs in $\varpi\big( P^{(2)}_{k,\ell;T,T'} \big)$, then it contains a scalar $K$\nbd type. Hence it is a holomorphic discrete series by Proposition~\ref{prop:irreducible-unitary-walls-scalar-K-type}. This completes the proof.
\end{proof}


\renewbibmacro{in:}{}
\renewcommand{\bibfont}{\normalfont\small\raggedright}
\renewcommand{\baselinestretch}{.8}

\Needspace*{4em}
\begin{multicols}{2}
\printbibliography[heading=none]
\end{multicols}



\Needspace*{3em}
\noindent
\rule{\textwidth}{0.15em}

{\noindent\small
Kathrin Bringmann\\
Mathematisches Institut, 
Universit\"at zu K\"oln, 
Weyertal 86-90, D-50931 K\"oln,
Germany\\
E-mail: \url{kbringma@math.uni-koeln.de} 
}

\vspace{2ex}

{\noindent\small
Olav K. Richter\\
Department of Mathematics,
University of North Texas,
Denton, TX 76203,
USA\\
E-mail: \url{richter@unt.edu}
}

\vspace{2ex}

{\noindent\small
Martin Westerholt-Raum\\
Chalmers tekniska högskola och G\"oteborgs Universitet,
Institutionen för Matematiska vetenskaper,
SE-412 96 Göteborg, Sweden\\
E-mail: \url{martin@raum-brothers.eu}\\
}

\end{document}
